\documentclass[a4paper]{amsart}
\usepackage{amsmath,amsthm,amssymb}
\usepackage{amscd}
\usepackage{mathrsfs}
\usepackage[all]{xy}
\usepackage{graphicx}
\usepackage{color}

\newtheorem{thm}{Theorem}[section]
\newtheorem{lem}[thm]{Lemma}

\newtheorem{prop}[thm]{Proposition}

\theoremstyle{definition}

\newtheorem{defn}[thm]{Definition}

\newtheorem{Rem}[thm]{Remark}

\theoremstyle{remark}

\numberwithin{equation}{subsection}

\def\XXint#1#2#3{{\setbox0=\hbox{$#1{#2#3}{\int}$}
\vcenter{\hbox{$#2#3$}}\kern-.5\wd0}}

\newcommand{\R}{\mathbb{R}}

\newcommand{\bbC}{\mathbb{C}}

\newcommand{\bbR}{\mathbb{R}}

\newcommand{\bbZ}{\mathbb{Z}}

\newcommand{\frakh}{\mathfrak{h}}
\newcommand{\frakt}{\mathfrak{t}}

\newcommand{\vphi}{\varphi}
\newcommand{\grad}{{\rm grad}}

\usepackage[abbrev,nobysame]{amsrefs}

\makeatletter
\def\@makefnmark{%
\leavevmode
\raise.9ex\hbox{\check@mathfonts
\fontsize\sf@size\z@\normalfont%
\@thefnmark}%
}
\makeatother

\title[The generalized Pythagorean theorem on the compactifications]{The generalized Pythagorean theorem on the compactifications of certain dually flat spaces via toric geometry}
\author{Hajime Fujita}
\address[H.~Fujita]{Japan Women's University}
\email{fujitah@fc.jwu.ac.jp}


\begin{document}
\maketitle
\begin{abstract}
In this paper we study dually flat spaces arising from Delzant polytopes equipped with a symplectic potential together with their corresponding toric K\"ahler manifolds as their torifications. 
 We introduce a dually flat structure and the associated Bregman divergence on the boundary from the viewpoint of toric K\"ahler geometry.  
 We show a continuity and a generalized Pythagorean theorem for the divergence on the boundary. 
 We also provide a characterization for a  toric K\"ahler manifold to become a torification of a mixture family on a finite set.  
\end{abstract}
 
 \tableofcontents
 
 %
 %
 \section{Introduction}
 
 The notion of {\it dually flat space} also known as the {\it Hesse manifold} gives a unified perspective through a geometrical approach to various fields including statistical science and information theory. 
Any strictly convex function on a convex domain of affine space determines a dually flat space, which includes several important examples such as the {\it exponential family}. 
In this context, the convex function and its Hessian are referred to as the potential function and the Hesse metric, respectively.   
The Hesse metric is equal to the {\it Fisher metric} for the exponential family. 
Another important quantity is the {\it divergence} of the dually flat space. 
One example is the {\it Bregman divergence} which is defined in the context of {\it Legendre duality}. 
The divergence can be considered as a kind of square of distance on a Riemannian manifold. 
One fundamental property of divergence is the {\it generalized Pythagorean theorem} which is  important for applications in statistical inference.

Given a dually flat space, we can associate a K\"ahler structure on its tangent bundle in a way that we call Dombrowski's construction (\cite{Dombrowski}). 
 Recently Molitor (\cite{Molitor}) investigated a relation between dually flat spaces and {\it toric K\"ahler manifolds} by introducing the notion of {\it torification}. 
 The theory of toric manifold was originally developed in algebraic geometry, although it has become clear it is also interesting object in the framework of Hamiltonian torus action in symplectic geometry.  It has been actively studied as a theory of {\it toric symplectic manifolds/geometry}.
 Every compact toric symplectic manifold has natural K\"ahler structure, giving rise to a {\it toric K\"ahler manifolds/geometry}. 
 A remarkable feature in toric K\"ahler geometry is the fact that several properties and quantities on a compact toric K\"ahler manifold can be described by objects on the moment map image. 
 The image is a convex polytope which satisfies some integrality conditions and forms a class of polytopes called {\it Delzant polytopes}. 
 The description of the Riemannian metric on any compact toric K\"ahler manifold was obtained in \cite{Guillemin}, \cite{Abreu} and \cite{ApostolovCalderbankGauduchon}. 
 More precisely, the Riemannian metric is obtained as a Hesse metric on the interior of the Delzant polytope essentially through Dombrowski's construction. 
 The potential function is called a {\it symplectic potential}.  
 Conversely, for any Delzant polytope with a symplectic potential one has a dually flat space,  
and furthermore one can construct a compact toric K\"ahler manifold by applying the famous Delzant construction. 
 This implies that each Delzant polytope with symplectic potential has a natural torification. 
 
 In this paper we study the torification and the boundary behavior of the dually flat space arising from Delzant polytope with symplectic potential.  
 The Hesse metric itself cannot be extended to the boundary, however, one can see that the dually flat structure and the associated Bregman divergence can be naturally defined on the boundary by using basic property of toric K\"ahler geometry. See Definition~\ref{defn:boundary dually flat}. 
 We emphasize that we use {\it symplectic coordinate} in our computation. 
 This corresponds to the use of {\it mixture parameter} rather than {\it canonical parameter} in the case of exponential family, which gives us an advantage in capturing the boundary behavior. 
 We also provide a characterization for compact toric K\"ahler manifold to become a torification of a {\it mixture family} on a finite set.
Our main results in this paper are as follows. 


\medskip

\noindent
{\bf Main results(Theorem~\ref{prop:M_P to P}, \ref{thm :  conti divergence}, \ref{thm:Pythagorean theorem} and \ref{thm:Pythagorean theorem2}). }
\\
{\it
1.  A compact toric K\"ahler manifold gives a torification of a mixture family on a finite set if and only if a zero-sum condition (\ref{eq:zerosum}) is satisfied. 
\\
2. For a Delzant polytope with a symplectic potential and its Bregman divergence we have : 
\begin{itemize}
\item a continuity on the boundary and 
\item the generalized Pythagorean theorem that is also applicable to boundary points.
\end{itemize}
}

\noindent
The second result asserts that the divergence can be defined on the compactification of the dually flat space arising from any Delzant polytope so that the generalized Pythagorean theorem holds. 
The foundation for these results is Proposition~\ref{prop : Divergence of P}, which presents a formula for the Bregman divergence of a Delzant polytope with a symplectic potential. 

A typical example where our results is applicable is the standard toric action on $\bbC P^n$ with the corresponding standard $n$-simplex as a Delzant polytope. 
In the context of information geometry this corresponds to the family of probability density functions on the finite set of $(n+1)$-points, often called the {\it categorical distribution}. 
Typically, zero probability cases are not explicitly addressed, although,  our results at the boundary provide a framework that encompasses such cases.

Recently Nakajima-Ohmoto (\cite{Nakajima-Ohmoto}) gave a reformulation of dually flat spaces in terms of Legendre fibration and extended it so that one can handle a statistical model with degeneration of Fisher metric. 
It would be interesting to investigate how the torification or our construction can be interpreted within Nakajima-Ohmoto's formulation. 

The organization of this paper is as follows.  
In Section~2 we first give a brief review on toric geometry from a viewpoint of symplectic geometry. 
After that we give a description of Riemannian metric of torus invariant K\"ahler structure along \cite{Guillemin}, \cite{Abreu} and \cite{ApostolovCalderbankGauduchon}. 
In Section~3 we put together fundamental definitions and facts of dually flat spaces and their torification. 
In Section~4 we explain a dually flat structure on a Delzant polytope with a symplectic potential.  
We derive a formula for the Bregman divergence in Proposition~\ref{prop : Divergence of P}, which coincides essentially with the {\it Kullback-Leibler divergence} for a mixture family of probability density functions on a finite set. 
Theorem~\ref{prop:M_P to P} provides a characterization as a zero-sum condition (\ref{eq:zerosum}) for a compact toric K\"ahler manifold to become a torification of such a family. 
In Section~5 we investigate the boundary behavior of the dually flat space arising from Delzant polytope, and give proofs of main results, Theorem~\ref{thm :  conti divergence} and Theorem~\ref{thm:Pythagorean theorem}.  
In Section~6 we exhibit two examples. 
In the first example we demonstrate the verification of our main theorems using the isosceles triangle and $\bbC P^2$.  
The second example is a non-compact case, which suggests that our main results can be extended to non-compact case. 

\medskip

\noindent
{\bf Acknowledgement.}
The author is indebted to Koichi Tojo for introducing him to the reference \cite{Molitor} and providing valuable comments.  
He is also grateful to Mao Nakamura and Naomichi Nakajima for fruitful discussions on dually flat spaces. 
He has gratitude to Mathieu Molitor for giving useful comments on the draft version of this paper. 
The author is partly supported by Grant-in-Aid for Scientific Research (C) 18K03288.

 \section{K\"ahler structure on toric symplectic manifold}
 In this section we give brief review on K\"ahler metrics on toric symplectic manifold in terms of Delzant polytopes and 
 symplectic potentials.

\subsection{Delzant polytope and Delzant construction}
For the contents of this subsection consult the book \cite{Silva} for example. 

\begin{defn}\label{defn:toric symp mfd}
A {\it toric symplectic manifold} $(M,\omega, \rho, \mu)$ is a data consisting of 
\begin{itemize}
\item a connected symplectic manifold $(M,\omega)$ of dimension $2n$, 
\item a homomorphism $\rho$ from the $n$-dimensional torus $T$ to the group of symplectomorphisms of $M$ which gives a Hamiltonian action of $T$ on $M$, and 
\item a moment map $\mu:M\to \bbR^n=({\rm Lie}(T))^*$. 
\end{itemize}
\end{defn}
In this paper, unless otherwise stated, we assume that toric manifolds are compact without boundary. 
However, in Section 6, we will consider non-compact examples.
For simplicity, we often denote the tuple $(M,\omega, \rho,\mu)$ as simply $M$. 
It is well known after \cite{Delzant} that the image $\mu(M)$ is a convex polytope in $\bbR^n$, which is called a {\it Delzant polytope} defined below. 

\begin{defn}\label{def:delzantpoly}
A convex polytope $P$ in $\R^n$ is called a {\it Delzant polytope} if $P$ satisfies the following conditions : 
\begin{itemize}
\item $P$ is simple, that is, each vertex of $P$ has exactly $n$ edges. 
\item $P$ is rational, that is, at each vertex all directional vectors of edges can be taken as integral vectors in $\bbZ^n$. 
\item $P$ is smooth, that is, at each vertex we can take integral directional vectors of edges as a $\bbZ$-basis of $\bbZ^n$ in $\bbR^n$. 
\end{itemize}
\end{defn}

Consider the group of integral affine transformations ${\rm AGL}(n,\bbZ)$. 
Namely ${\rm AGL}(n,\bbZ)={\rm GL}(n,\bbZ)\times \bbR^n$ as a set and the group structure is given by 
\[
(A_1, v_1)\cdot(A_2, v_2):=(A_1A_2, A_1v_2+v_1)
\]for each $(A_1, v_1), (A_2, v_2)\in {\rm AGL}(n,\bbZ)$. 
This group ${\rm AGL}(n,\bbZ)$ naturally acts on the set of all $n$-dimensional Delzant polytopes. 
The famous Delzant construction gives the inverse correspondence of $M\mapsto \mu(M)$ for a compact toric symplectic manifold $M$ without boundary. 
To give the precise statement we clarify the equivalence relation. 

\begin{defn}
Two toric symplectic manifolds $(M_1, \omega_1, \rho_1,\mu_1)$ and $(M_2, \omega_2, \rho_2, \mu_2)$ are {\it weakly isomorphic} if there exist a diffeomorphism $f:M_1\to M_2$ and a group isomorphism $\phi:T\to T$ such that 
\[
f^*\omega_2=\omega_1 \ {\rm and} \ \rho_1(g)(x)=\rho_2(\phi(g))(f(x)) \ {\rm for\ all }  \ (g,x)\in T\times M_1. 
\]
\end{defn}
\noindent 
In \cite{KarshonKessler} the equivalence relation \lq\lq weakly isomorphic \rq\rq \ is called 
just \lq\lq  equivalent \rq\rq. In this paper we follow the terminology in \cite{PPRS}.
 
\begin{thm}[\cite{Delzant},\cite{KarshonKessler}]
The Delzant construction gives a bijective correspondence between the set of all ${\rm AGL}(n,\bbZ)$-congruence classes of Delzant polytopes and the set of all weak isomorphism classes of $2n$-dimensional compact toric symplectic manifolds without boundary. 
\end{thm}

We give a brief review of the Delzant construction here. 
Let $P$ be an $n$-dimensional Delzant polytope and
\begin{equation}\label{l^r}
l^{(r)}(\xi):=\xi\cdot\nu^{(r)} + \lambda^{(r)}=0 \quad (r=1,\cdots, N)
\end{equation}
a system of defining affine equations on $\bbR^n$ of facets of $P$, 
where each $\nu^{(r)}\in\bbZ^n$ being inward pointing normal vector of $r$-th facet, $\cdot$ is the Euclidean inner product and $N$ is the number of facets of $P$. In other words, $P$ can be described as 
\[
P=\bigcap_{r=1}^N\{\xi\in\bbR^n \ | \ l^{(r)}(\xi)\geq 0\}. 
\]
 We may assume that each $\nu^{(r)}$ is primitive and the set of all $\nu^{(r)}$'s at each vertex form a $\bbZ$-basis of $\bbZ^n$. 
Here an integral vector $u$ in $\bbR^n$ is called {\it primitive} if $u$ cannot be described as $u=ku'$ for another integral vector $u'$ and $k\in\bbZ$ with $|k|>1$. 
Consider the standard Hamiltonian action of the $N$-dimensional torus $T^N$ on $\bbC^N$ with the moment map  
\[
\tilde\mu:\bbC^N\to (\bbR^N)^*={\rm Lie}(T^N)^*, \ (z_1, \ldots, z_N)\mapsto  (|z_1|^2, \ldots, |z_N|^2)+(\lambda^{(1)}, \ldots, \lambda^{(N)}). 
\]
Let $\tilde\pi:\bbR^N\to \bbR^n$ be the linear map defined by $e_r\mapsto\nu^{(r)}$, where $e_r$ ($r=1,\ldots, N$) is the 
$r$-th standard basis vector of $\bbR^N$. 
Note that $\tilde\pi$ induces a surjection $\tilde\pi=\tilde\pi|_{\bbZ^N}:\bbZ^N\to\bbZ^n$ between the standard lattices by the 
last condition in Definition~\ref{def:delzantpoly}, and hence it induces a surjective homomorphism between tori,  still denoted by $\tilde\pi$, 
\[
\tilde\pi:T^N=\bbR^N/\bbZ^N\to T=\bbR^n/\bbZ^n. 
\]
Let $H$ be the kernel of $\tilde\pi$ which is an $(N-n)$-dimensional subtorus of $T^N$ and 
$\frakh$ its Lie algebra.  
We have an exact sequence 
\[
1\to  H\stackrel{\iota}{\to} T^N\stackrel{\tilde\pi}{\to} T\to 1
\]
and the associated exact sequence among Lie algebras 
\[
0 \to  \frakh\stackrel{\iota}{\to} \bbR^N\stackrel{\tilde\pi}{\to} \bbR^n\to 0, 
\]
where $\iota$ is the inclusion map from $H$ (resp. $ \frakh$) to $T^N$ (resp. $\R^N$). 
By taking the dual of the exact sequence among Lie algebras we have the exact sequence
\[
0 \to (\bbR^n)^* \stackrel{\tilde\pi^*}{\to} (\bbR^N)^*\stackrel{\iota^*}{\to}  \frakh^*\to 0. 
\]
Then the composition $\iota^*\circ\tilde\mu:\bbC^N\to \frakh^*$ is the associated moment map 
of the action of $H$ on $\bbC^N$. 
It is known that $(\iota^*\circ\tilde\mu)^{-1}(0)$ is a compact submanifold of $\bbC^N$ and 
$H$ acts freely on it.  
We obtain the desired symplectic manifold, denoted as $M_P$, defined as the quotient $(\iota^*\circ\tilde\mu)^{-1}(0)/H$, equipped with a natural Hamiltonian $T^N/H=T$-action.  
Its moment map image precisely coincides with $P$. 
From the viewpoint of the moment map, it is natural to regard $P$ as a subset of ${\rm Lie}(T)^*$ and 
each normal vector $\nu^{(r)}$ as an element in ${\rm Lie}(T)^{**}={\rm Lie}(T)$. 

\subsection{The Guillemin metric and Abreu's theory}
Note that the standard flat K\"ahler structure on $\bbC^N$ induces a K\"ahler structure on $M_P$. 
The associated Riemannian metric on $M_P$ is called the {\it Guillemin metric}. 
An explicit description of the Guillemin metric is known.  
We give the description following \cite{Abreu}. 
Consider the smooth function 
\begin{equation}\label{g_P}
\vphi_P:=\frac{1}{2}\sum_{r=1}^Nl^{(r)}\log l^{(r)} : P^{\circ}\to \bbR, 
\end{equation}
where $P^{\circ}$ is the interior of $P$. 
It is known that $M_P^{\circ}:=\mu_P^{-1}(P^{\circ})$  is an open dense subset of $M_P$ which coincides with the union of free $T^n$-orbits. 
By taking a Lagrangian section $P^\circ \to M_P^\circ$ we have a trivialization $M_P^{\circ}\cong P^{\circ}\times T$ as a principal torus bundle. 
Under this identification $\omega_P|_{M_P^{\circ}}$ can be described as 
\[
\omega_P|_{M_P^{\circ}}=d\xi\wedge dt=\sum_{i=1}^nd\xi_i\wedge dt_i
\]using the standard coordinate $(\xi,t)=(\xi_1, \ldots, \xi_n, t_1, \ldots, t_n)\in P^{\circ}\times T$. 
Here we regard $T=(S^1)^n$ and $S^1=\bbR/\bbZ$. The coordinate on $M_P^{\circ}$ induced from $(\xi,t)\in P^{\circ}\times T$ is called the {\it symplectic coordinate} on $M_{P}$. 

\begin{thm}[\cite{Guillemin}]\label{Guillemins metric}
Under the symplectic coordinates $(\xi,t)\in P^{\circ}\times T\cong M_P^{\circ}\subset M_P$, the Guillemin metric can be described as 
\[
\begin{pmatrix}
G_P & 0 \\
0 & G_P^{-1}
\end{pmatrix}, 
\]where $\displaystyle G_P:={\rm Hess}(\vphi_P)=\left(\frac{\partial^2 \vphi_P}{\partial \xi_i\partial \xi_j}\right)_{i,j=1,\ldots, n}$ is the Hessian of $\vphi_P$. 
\end{thm}
Furthermore the complete classification of torus invariant K\"ahler structures is known. 
To explain it we introduce a class of functions on $P$. 
\begin{defn}
A smooth function $\vphi:P^\circ\to \bbR$ is called a {\it symplectic potential} if the following conditions are satisfied. 
\begin{itemize}
\item $G:={\rm Hess}(\vphi)$ is positive definite. 
\item $\vphi-\vphi_P$ has a smooth extension to an open neighborhood of $P$. 
\item $\displaystyle\det(G)\left(\prod_{r=1}^Nl_r\right)$ is positive and has a smooth extension to an open neighborhood of $P$. 
\end{itemize}
\end{defn}
The first condition of the symplectic potential tells us that 
\[
\begin{pmatrix}
G& 0 \\
0 & G^{-1}
\end{pmatrix} \quad ({\rm resp.}  \ G)
\]induces a torus invariant K\"ahler (resp. Riemannian) structure on $M_P^\circ$ (resp. $P^\circ$). 
The following theorem asserts the converse and provides a classification of torus invariant K\"ahler structures. 
\begin{thm}[\cite{Abreu}, \cite{ApostolovCalderbankGauduchon}]\label{thm : toric Kahler metric}
For any symplectic potential $\vphi:P^\circ\to \bbR$ and $G:={\rm Hess}(\vphi)$ the Riemannian metric
\begin{equation}\label{eq : local description of metric}
\begin{pmatrix}
G& 0 \\
0 & G^{-1}
\end{pmatrix} 
\end{equation}on $M_P^\circ$ defines a  torus invariant K\"ahler structure on $M_P$. 
Conversely for any compact toric K\"ahler manifold $M$ with the moment map $\mu:M\to \bbR^n$ and its Riemannian metric $g$ there exists a symplectic potential $\vphi:\mu(M^\circ)\to \bbR$ such that $g|_{M^\circ}$ can be represented as (\ref{eq : local description of metric}) for $G={\rm Hess}(\vphi)$. 
\end{thm}

 \section{Dually flat spaces and torifications}
 In this section we prepare several basic notions of dually flat spaces. 
 \subsection{Dually flat spaces and Hesse structures}
 For the contents of this and next subsections consult the book \cite{Shima} for example. 
 \begin{defn}\label{defn : dually flat space}
 A {\it dually flat space} is a triplet $(U,h,\nabla)$ consisting of 
 \begin{itemize}
 \item a smooth manifold $U$, 
 \item a torsion free flat affine connection $\nabla$ of $U$, and 
 \item a Riemannian metric $h$ of $U$ such that $\nabla h$ is totally symmetric. 
 \end{itemize}
 \end{defn}
 
 \begin{Rem}
 For any affine connection $\nabla$ of a Riemannian manifold $(U,h)$, the {\it dual connection $\nabla^*$} of $\nabla$ is determined by the equality 
 \[
 Xh(Y,Z)=h(\nabla_XY,Z)+h(Y, \nabla^*_XZ)
 \]for any vector fields $X, Y, Z$ on $U$.  
 Then $(U,h,\nabla)$ is dually flat if and only if $(U,h,\nabla^*)$ is so. 
 \end{Rem}
 
 \begin{defn}\label{defn : affine coord}
 Let $(U,h,\nabla)$ be an $n$-dimensional dully flat space, and let $u=(u_1,\ldots, u_n)$ and $v=(v_1,\ldots, v_n)$ be coordinate systems on $U$. The pair $(u,v)$ is called a {\it pair of dual affine coordinate systems} if the following conditions are satisfied. 
 \begin{itemize}
 \item $u$ is a $\nabla$-affine coordinate, i.e., $\nabla_{\frac{\partial}{\partial u_i}}\frac{\partial}{\partial u_j}=0$ for all $i,j=1,\ldots, n$. 
 \item $v$ is a $\nabla^*$-affine coordinate, i.e., $\nabla^*_{\frac{\partial}{\partial v_i}}\frac{\partial}{\partial v_j}=0$ for all $i,j=1,\ldots, n$. 
 \item $h\left(\frac{\partial}{\partial u_i}, \frac{\partial}{\partial v_j}\right)=\delta_{ij}$ for all $i, j=1,\ldots, n$. 
 \end{itemize}
 \end{defn}
 
 The notion of dual affine coordinate systems can be defined for a local coordinate system of $U$, however, our examples in the subsequent sections can be described by global coordinate.
 
 \begin{prop}\label{prop : potential}
Let $(U,h,\nabla)$ be an $n$-dimensional dually flat space and $u:U\to \R^n$ a $\nabla$-affine coordinate system. Suppose that $U$ is connected, simply connected and $u(U)\subset \bbR^n$ is convex. Then there exists a smooth function $\varphi :u(U)\to \R$ such that the following conditions are satisfied. 
\begin{itemize}
\item $\displaystyle\frac{\partial^2\varphi}{\partial u_i\partial u_j}=h\left(\frac{\partial}{\partial u_i}, \frac{\partial}{\partial u_j}\right)$
\item The pair $(u,v)$ is a pair of dual affine coordinate systems, where $v=(v_1, \ldots,
v_n)={\rm grad}(\vphi)=\left(\frac{\partial\vphi}{\partial u_1}, \ldots, \frac{\partial\vphi}{\partial u_n}\right)$. 
\end{itemize}
 \end{prop}
 
In the set-up in Proposition~\ref{prop : potential}, the smooth function $\vphi$ is referred to as a {\it potential} of the metric $h$, and $(h={\rm Hesse}(\vphi), \nabla)$ is refereed to as a {\it Hesse structure}.   
The following lemma serves as a fundamental tool in this paper and can be verified through the direct computation. 

\begin{lem}\label{lem : affine dually flat}
Let $U$ be a convex open subset of $\R^n$. 
Suppose that a Hesse metric $h$ on $U$ is given. Namely $h$ can be written as  $h=\nabla^{\rm flat}(d\vphi)={\rm Hess}(\vphi)$ for some smooth function $\vphi:U\to \bbR$, where $\nabla^{\rm flat}$ is the canonical flat connection on $U$ with respect to the affine structure. 
Then $(U,h, \nabla^{\rm flat})$ is a dually flat space with an affine coordinate $u={\rm id}:U\to U\subset \R^n$. \end{lem}
\subsection{Bregman divergence of dually flat spaces}
Let $(U,h,\nabla)$ be a dually flat space. 
Suppose that there exists a pair of dual affine coordinate systems $(u,v)$ on $U$ and a potential $\vphi$ for $(u,v)$.  
Then we have the {\it Legendre dual {\rm or the} dual potential $\psi : v(U)\to \bbR$} of $\vphi$, which is characterized as 
\[
\vphi(u)+\psi(v)-u\cdot v=0. 
\]
Namely $\psi:v(U)\to \bbR$ is defined as 
\[
\psi(v(\xi))=-\vphi(u(\xi))+u(\xi)\cdot v(\xi)
\]for $\xi\in U$ under the identification $u:U\cong u(U)$ and $v:U\cong v(U)$. 
Using the Legendre dual one can define the Bregman divergence of the Hesse structure. 
\begin{defn}
For a dually flat space $(U,h,\nabla)$ with a pair of dual affine coordinate systems $(u,v)$ and a potential $\vphi:u(U)\to \bbR$, 
the {\it Bregman divergence} $D(\cdot \| \cdot):U\times U\to\bbR$ is defined by 
\[
D(\xi \| \xi')=\vphi(u(\xi))+\psi(v(\xi'))-u(\xi)\cdot v(\xi'). 
\]
\end{defn}
Note that the divergence is independent of the choice of the pair of dual affine coordinate systems and potentials. 

\subsection{Torification of dually flat spaces}
Following \cite{Molitor} we introduce the notion of torification of dually flat spaces. 
Let $(U,h,\nabla)$ be an $n$-dimensional dually flat space. 
Recall that by Dombrowski's construction (\cite{Dombrowski}) the tangent bundle $TU$ has a natural K\"ahler structure. 
Suppose that there exists a {\it parallel lattice} $\Gamma$ with respect to $\nabla$ generated by $n$-tuple of parallel vector fields $(X_1, \ldots, X_n)$ on $U$. Namely :  
\begin{itemize}
\item $\{X_1(\xi),\ldots, X_n(\xi)\}$ is a basis of $T_\xi U$ for each $\xi\in U$. 
\item $\Gamma=\{k_1X_1(\xi)+\cdots+ k_nX_n(\xi) \ | \ \xi\in U, k_1,\ldots, k_n\in \bbZ\}$. 
\end{itemize}
Note that $\Gamma$ is isomorphic to $\bbZ^n$ as an abelian group. 
The lattice $\Gamma$ acts on $TU$ in an effective and holomorphic isometric way by the translation: 
\[
T_\xi U\ni v \mapsto v+k_1X_1(\xi)+\cdots+ k_nX_n(\xi)\in T_\xi U
\]  for $\xi\in U$ and $k_1,\ldots, k_n\in \bbZ$. 
Using this group action we have a K\"ahler manifold as the quotient $U_\Gamma:=TU/\Gamma$. 
Furthermore, the natural projection $\pi_\Gamma:U_\Gamma\to U$ has a structure of a {\it Lagrangian torus fibration} and is endowed with a natural $T=(S^1)^n$-action, which induces a structure of a (possibly non-compact) toric K\"ahler manifold on $U_\Gamma$.   

\begin{defn}
Let $M$ be a K\"ahler manifold equipped with a holomorphic isometric action of $T$.  
$M$ is said to be a {\it torification} of a dually flat space $(U,h,\nabla)$ if there exists a parallel lattice $\Gamma$ of $U$ and a $T$-equivariant isomorphism between the K\"ahler manifolds $U_\Gamma$ and $M^\circ$, where $M^\circ$ is the set of points in $M$ on which $T$ acts freely. 
\end{defn}

\begin{Rem}We give small remarks on torifications. 
\begin{enumerate}
\item If a dually flat space $(U,h,\nabla)$ has a parallel lattice $\Gamma$, then $U_\Gamma$ itself is obviously a torification of it. For instance, the dually flat space described in Lemma~\ref{lem : affine dually flat} has a natural torification. 
\item A torification is not necessarily compact. In fact, the torification $U_\Gamma$ mentioned above is not compact in general. 
\item In \cite{Molitor} the notion of {\it regular torification} is defined as a special class of torifications and studied in detail. In particular the uniqueness of the regular torification is established. 
\end{enumerate}
\end{Rem}

\begin{Rem}
In contrast to the compact case, it is known in \cite{KarshonLerman} that non-compact toric manifolds cannot be classified by their moment map image, and additional data called the characteristic classes are needed in general.  
As it was shown in \cite[Proposition~6.5.]{KarshonLerman}, if the moment map is a proper map into a convex open subset in the dual of the Lie algebra, then the toric manifold is determined by the moment map image up to isomorphism even if it is non-compact. 
All non-compact toric manifolds treated in the present paper as torifications satisfy this condition.  
\end{Rem}



 \section{Delzant polytopes as dually flat spaces}
 Let $P$ be an $n$-dimensional Delzant polytope, $\vphi:P^\circ\to \R$ a symplectic potential and $T$ an  $n$-dimensional torus with ${\rm Lie}(T)=\frakt\cong \bbR^n$. 
 Since its Hessian $G:={\rm Hess}(\vphi)$ is positive definite, $G$ determines a structure of a dually flat space on $P^\circ$ by Lemma~\ref{lem : affine dually flat}. 
 We explain its torification following \cite{Molitor}.

  \subsection{Torification of Delzant polytopes}
  \label{subsec:Torification of Delzant polytope}
 We use {\it complex coordinate} on toric K\"ahler manifold to describe the torification of a Delzant polytope. 
 
 Let $M_P$ be the compact toric K\"ahler manifold associated with $P$ and $\vphi$.  
 As it is explained in \cite{Abreu}, it is known that there is an isomorphism of K\"ahler manifolds 
 \[
 P^\circ\times T\cong M_P^\circ\cong\bbR^n\times T
 \]given by $(\xi, t)\mapsto (y(\xi), t)=(\grad(\vphi)(\xi), t)$, 
 where the K\"ahler structure on the right hand side is given by 
 \[
 {\rm complex \ structure}=\begin{pmatrix}
 0 & -{\rm id} \\ 
 {\rm id} & 0
 \end{pmatrix},
 \]
 \[
 {\rm symplectic \ structure}=\begin{pmatrix}
 0 & {\rm Hess}(\psi) \\ 
 -{\rm Hess}(\psi) & 0
 \end{pmatrix}, 
 \] and
\[ 
 {\rm  Riemannian \ metric}=
 \begin{pmatrix}
 {\rm Hess}(\psi) & 0 \\ 
 0 & {\rm Hess}(\psi)
 \end{pmatrix}
 \]for the Legendre dual $\psi$ of $\vphi$. 
 Let $x^*$ be the standard coordinate on $\bbR^n$.
 The coordinate $(x^*,t)$ on $M_P^\circ\cong \bbR^n\times T$ is often called the {\it complex coordinate} in the context of toric K\"ahler geometry. 
 In fact $\bbR^n\times T$ is isomorphic to the $n$-dimensional complex torus $(\bbC^\times)^n$ by $(x^*, t)\mapsto e^{x^*+it}$, and it is isomorphic to the open dense $(\bbC^\times)^n$-orbit in $M_P$.  

We have a dually flat space $(\bbR^n, {\rm Hess}(\psi), \nabla^{\rm flat})$ with the dual pair of affine coordinate systems $(x^*, y^*=\grad(\psi))$ and a parallel frame $\Gamma$ of $T\bbR^n$ generated by $\left(\frac{\partial}{\partial x^*_i}\right)_{i=1,\ldots,n}$. 
Then the above K\"ahler structure on $\bbR^n\times T$ is nothing other than that on $(\bbR^n)_\Gamma$ using Dombrowski's construction. 
This implies that $M_P$ is a torification of $(\bbR^n, {\rm Hess}(\psi), \nabla^{\rm flat})$.  
In fact \cite[Theorem~7.1]{Molitor} shows that $y^* : \bbR^n\to P^\circ$ gives an isomorphism between dually flat spaces $(\bbR^n, {\rm Hess}(\psi), \nabla^{\rm flat})$ and $(P^\circ, G, \nabla=(\nabla^{\rm flat})^*)$, which is a Legendre transform.
In particular $M_P$ is a torification of $(P^\circ, G, \nabla)$. 

The affine coordinate $x^*$ on $(\R^n, {\rm Hess}(\psi), \nabla^{\rm flat})$ is called the {\it canonical parameter} of exponential family in information geometry. 
Let $x$ be the standard coordinate on $P^\circ \subset \frakt^*\cong \bbR^n$, which gives an affine coordinate 
on $(P^\circ, G,  \nabla^{\rm flat}=\nabla^*)$. 
The affine coordinate $x$ is called the {\it mixture parameter} or {\it expectation parameter}. 
In \cite{Molitor} the dually flat space $(\bbR^n, {\rm Hess}(\psi), \nabla^{\rm flat})$ is studied in detail together with its dual space $(P^\circ, G,\nabla)$. 
Hereafter we mainly focus on the dually flat space $(P^\circ, G,  \nabla^{\rm flat})$ with the dual pair of affine coordinate systems $(x,y=\grad(\vphi))$ to capture the boundary behavior.

\begin{Rem}\label{rem:toric to dually flat}
We also have the correspondence of the converse direction. 
Namely  for any compact toric K\"ahler manifold $M$ with the moment map $\mu:M\to \bbR^n$ we have a symplectic potential $\vphi:\mu(M)^\circ\to \bbR$ by the latter part of Theorem~\ref{thm : toric Kahler metric}. 
It determines a dually flat space $(\mu(M^\circ), {\rm Hess}(\vphi), \nabla^{\rm flat})$ by Lemma~\ref{lem : affine dually flat}. 
\end{Rem}

 \subsection{Bregman divergence of Delzant polytopes}
We derive a formula of the Bregman divergence of a Delzant polytope with a symplectic potential, which is a fundamental tool for the subsequent sections in this paper.  
 
 \begin{prop}\label{prop : Divergence of P}
 We have the following formula for the Bregman divergence of $(P^\circ, G={\rm Hess}(\vphi),\nabla^{\rm flat})$. 
 \[
 D(\xi \| \xi')=\frac{1}{2}\sum_{r=1}^{N}\left(l^{(r)}(\xi)\log\frac{l^{(r)}(\xi)}{l^{(r)}(\xi')}-(\xi-\xi')\cdot\nu^{(r)}\right)
 +(\xi'-\xi)\cdot (\grad(f)(\xi'))+f(\xi)-f(\xi'), 
 \]where $f:=\vphi-\vphi_P$. 
 \end{prop}
 \begin{proof}
 We start from the canonical affine coordinate $x={\rm id} : P^\circ\to \R^n$. 
 Its dual affine coordinate and dual potential are given by 
 \[
 y={\rm grad}(\vphi)=\left(\frac{\partial \vphi}{\partial x_1},\cdots, \frac{\partial \vphi}{\partial x_n}\right): P^\circ\to \R^n
 \]
 and 
 \[
 \psi(y(\xi))=-\vphi(x(\xi))+x(\xi)\cdot y(\xi)=-\vphi(\xi)+\xi\cdot (\grad(\vphi)(\xi)). 
 \]
 By using the formula 
 \[
 \grad(\vphi)=\grad(\vphi_P)+\grad(f)=\frac{1}{2}\sum_r\left(\nu^{(r)}\log l^{(r)}+\nu^{(r)}\right)+\grad(f), 
 \]
 the Bregman divergence can be computed as  
 \begin{eqnarray*}
 D(\xi \| \xi')&=& \vphi(x(\xi))+\psi(y(\xi'))-x(\xi)\cdot y(\xi') \\ 
 &=& \vphi(\xi)-\vphi(\xi')+\xi'\cdot (\grad(\vphi)(\xi')) -\xi\cdot(\grad(\vphi)(\xi')) \\ 
 &=& \frac{1}{2}\sum_r\left(l^{(r)}(\xi)\log l^{(r)}(\xi)-l^{(r)}(\xi')\log l^{(r)}(\xi')+(\xi'-\xi)\cdot (\nu^{(r)}\log l^{(r)}(\xi')+\nu^{(r)})\right) \\ 
 && \hspace{6cm}+(\xi'-\xi)\cdot (\grad(f)(\xi'))+f(\xi)-f(\xi') \\
 &=& \frac{1}{2}\sum_{r}\left(l^{(r)}(\xi)\log\frac{l^{(r)}(\xi)}{l^{(r)}(\xi')}-(\xi-\xi')\cdot\nu^{(r)}\right)+(\xi'-\xi)\cdot (\grad(f)(\xi'))+f(\xi)-f(\xi'). 
 \end{eqnarray*}
 \end{proof}
 
 \subsection{Torification of mixture families on a finite set}\label{subsec : prob densi func}
 Let $P$ be a Delzant polytope with $N$ defining inequalities $l^{(r)}(\xi)=\xi\cdot\nu^{(r)}+\lambda^{(r)}\geq 0$ ($r=1,\ldots, N$). 
 If normal vectors $\nu^{(1)},\cdots,  \nu^{(N)}$ satisfy the condition 
 \begin{equation}\label{eq:zerosum}
 \sum_{r}\nu^{(r)}={\bf 0}, 
 \end{equation}then we have 
 \[
 \sum_rl^{(r)}(\xi)=\sum_r\lambda^{(r)}
\]is constant. This implies that the Delzant polytope $P$ associates a family of probability density functions $\Theta_P=\{p(\cdot|\xi)\}_{\xi\in P^\circ}$ on a finite set $[N]:=\{1, \cdots, N\}$ defined by 
\[
p(r|\xi):=\frac{l^{(r)}(\xi)}{\sum_{r'}\lambda^{(r')}} \quad (r\in \{1,\cdots, N\}). 
\]
For the Guillemin potential $\varphi_P=\frac{1}{2}\sum_r l^{(r)}\log l^{(r)}$ of $P$, by Proposition~\ref{prop : Divergence of P} one has the Bregman divergence 
\[
D_P(\xi \| \xi')=\frac{1}{2}\sum_r\left(l^{(r)}(\xi)\log\frac{l^{(r)}(\xi)}{l^{(r)}(\xi')}\right). 
\]This is nothing other than the constant multiple of the {\it Kullback-Leibler divergence} $D_{\Theta_P}^{\rm KL}$ of $\Theta_P$, which is defined by 
\[
D_{\Theta_P}^{\rm KL}(\xi\|\xi')=\sum_{r}p(r|\xi)\log\frac{p(r|\xi)}{p(r|\xi')}=\frac{1}{\sum_{r'}\lambda^{(r')}}\sum_r\left(l^{(r)}(\xi)\log\frac{l^{(r)}(\xi)}{l^{(r)}(\xi')}\right). 
\]
This observation implies that the toric K\"ahler manifold $M_P$ equipped with the Guillemin metric is a  torification of $(P^\circ, G_P={\rm Hess}(\varphi_P),(\nabla^{\rm flat })^*)$ and $G_P$ is the constant multiple of the Fisher metric of the family of probability density functions $\Theta_P$. 

Now we investigate the converse. 
Consider a family of probability density functions $\Theta$ on the finite set $[N]=\{1, \cdots, N\}$ parametrized by an open subset $U$ of $\bbR^n$.  
Suppose that each element in $\Theta$ has a form 
\[
p(r|\xi)=\xi\cdot\alpha^{(r)} +\beta^{(r)} \quad  (\xi\in U, r\in[N])
\]for some $\alpha^{(r)}\in\bbR^n$, $\beta^{(r)}\in\bbR$. 
In other words, $\Theta$ is a {\it mixture family on the finite set $[N]$}. 
We assume that $U\ni \xi \mapsto p(\cdot|\xi)\in\Theta$ is bijective. 
Since we have the equalities 
\[
1=\sum_rp(r|\xi)=\sum_r\xi\cdot\alpha^{(r)}+\sum_r\beta^{(r)}
\]for all $\xi\in U$, we have the condition 
\begin{equation}\label{eq:zerosum2}
\sum_r\alpha^{(r)}={\bf 0}. 
\end{equation}
Let $\vphi:U\to \bbR$ be the function defined by 
\[
\vphi(\xi)=\frac{1}{2}\sum_rp(r|\xi)\log(p(r|\xi)), 
\]which is a (half of a) potential function of the Fisher metric on $\Theta$. 
For simplicity we may regard $U$ as the maximal open set on which all $p(r|\xi)$'s are defined. 
Namely we assume that 
\[
U=\{\xi\in\bbR^n \ | \ p(r|\xi)>0 \  (\forall r\in[N])\},  
\]then the closure $\overline{U}$ is a convex polytope in $\bbR^n$. 
If $\overline U$ is a Delzant polytope, then as we discussed in Section~\ref{subsec:Torification of Delzant polytope}, 
the toric K\"ahler manifold $M_{\overline U}$ equipped with the Guillemin metric is a torification of $(U,  {\rm Hess}(\vphi), \nabla=(\nabla^{\rm flat})^*)$.

\begin{lem}\label{lem:proper=compact}
If $M$ is a torification of $(U,{\rm Hess}(\vphi), \nabla=(\nabla^{\rm flat})^*)$, then  
the associated moment map $\mu:M\to \bbR^n$ is proper if and only if $M$ is compact. 
\end{lem}
\begin{proof}
It suffices to show that if $\mu$ is proper, then $M$ is compact. 
$\overline U$ can be embedded into $\Delta^N$ by 
\[
\overline U\ni \xi\mapsto (p(1|\xi), \ldots, p(N|\xi))\in \Delta^N, 
\]where $\Delta^N$ is the probability $N$-simplex in $\bbR^{N}$, 
\[
\Delta^N=\left\{(\eta_1, \ldots, \eta_{N})\in\bbR^{N} \ \middle| \ \sum_{j=1}^{N}\eta_j=1\right\}. 
\]
In particular $\overline U$ is compact. 

We may take a parallel lattice $\Gamma$ in $TU$ as the lattice generated by $\nabla$-parallel frame $\left(\frac{\partial}{\partial y_i}\right)_i$ for $y=\grad(\vphi)$. 
By definition there are an isomorphism 
\[
F:M^\circ \cong U_\Gamma = U\times T
\]between toric K\"ahler manifolds for $\Gamma$ of $(U,{\rm Hess}(\vphi), \nabla)$ and a compatible projection (\cite[Definition~6.6(3)]{Molitor}) $\kappa:M^\circ\to U$. 
By \cite[Proposition~6.8~(1)]{Molitor} we have 
\[
\mu|_{M^\circ}=x\circ\kappa={\rm pr}_1\circ F:M^\circ\to\bbR^n
\]for the restriction of the moment map $\mu$.  
Here the matrix of change of basis of $\Gamma$ in \cite[Proposition~6.8~(1)]{Molitor} is irrelevant. If necessary, we compose a translation to $\mu$ so that the constant vector $C$ in \cite[Proposition~6.8~(1)]{Molitor} does not appear.  
We have the equalities 
  \[
 \mu^{-1}(\overline U)=\overline{\mu^{-1}( U)}=\overline{M^\circ}=M,  
 \]and hence, since $\overline U$ is compact and $\mu$ is proper, $M$ is compact. 
\end{proof}

The following gives a characterization for a torification of our mixture family. 

\begin{thm}\label{prop:M_P to P}
Let $U$ be the above parameter space of the mixture family $\Theta$. 
If $(U,{\rm Hess}(\vphi), \nabla=(\nabla^{\rm flat})^*)$ has a compact torification $M$, then the closure $\overline{U}$ is the Delzant polytope associated with $M$ whose normal vectors of facets  satisfy the zero-sum condition (\ref{eq:zerosum}). 
\end{thm}

\begin{proof}
Since $M$ is compact the moment map $\mu$ is proper, and hence, together with the argument in the proof of Lemma~\ref{lem:proper=compact} we have the equalities
 \[
 \mu(M)=\mu(\overline{M^\circ})=\overline{\mu(M^\circ)}=\overline U.  
 \]
 The facets of the closure $\overline U$ consist of points in $\bbR^n$ satisfying 
 \[
\xi\cdot\alpha^{(r)} +\beta^{(r)} = 0. 
 \] 
  This condition and the relation~(\ref{eq:zerosum2}) imply the statement by multiplying the denominators of $\alpha^{(r)}$'s if necessarily.  
\end{proof}

\begin{Rem}
In the above set-up $\vphi$ induces the half of Fisher metric on $\Theta$ and it induces the Guillemin metric on $M_{\overline U}$. 
\end{Rem}

Theorem~\ref{prop:M_P to P} gives a condition for a toric K\"ahler manifold which can become a torification of a mixture family. 
For example a Hirzebruch surface whose moment map image is the trapezoid as in Figure~\ref{fig:trapezoid} cannot become a torification of any mixture family.

\begin{figure}[h]
\centering
\includegraphics[scale=0.3]{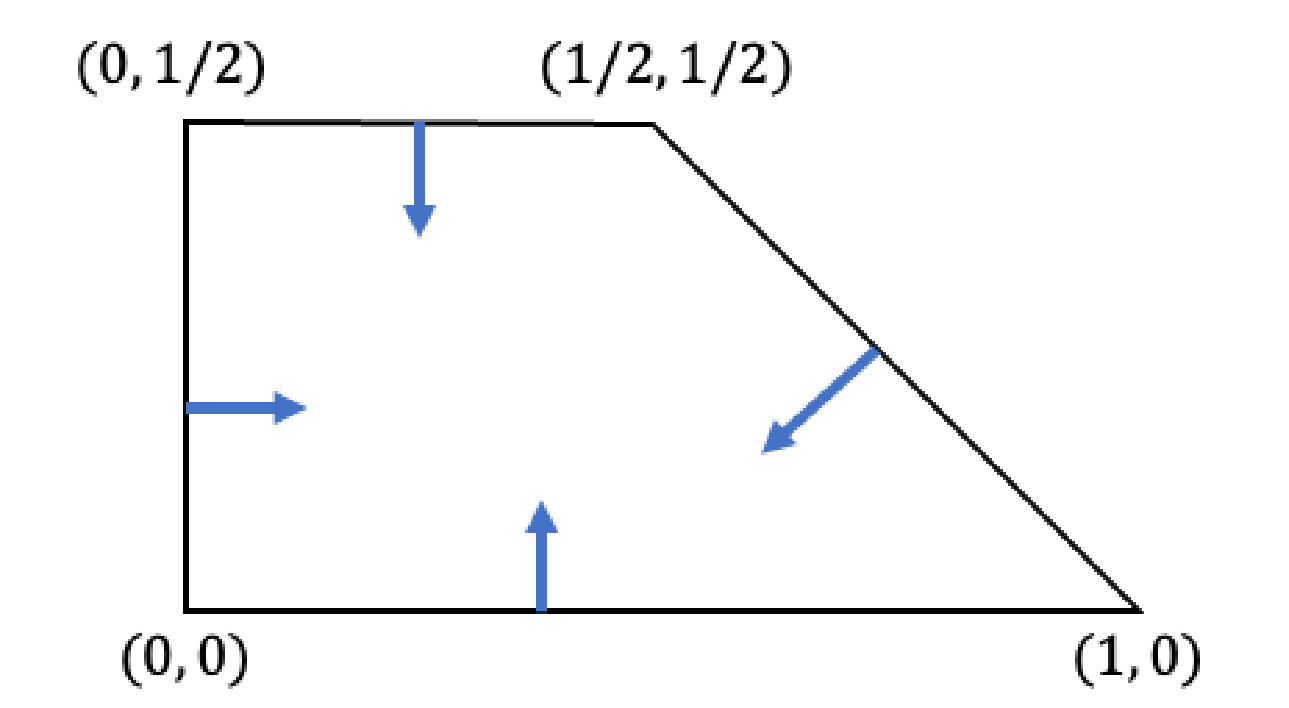}
\caption{A trapezoid of a Hirzebruch surface} \label{fig:trapezoid}
\end{figure}

 \section{Dually flat structure on the boundary of Delzant polytopes}
Let $P$ be an $n$-dimensional Delzant polytope and $T$ an  $n$-dimensional torus with ${\rm Lie}(T)=\frakt\cong \bbR^n$. 
We take and fix a $k$-dimensional face $F$ of $P$. 
Without loss of generality we may assume that $F$ is the set of solutions of $n-k$ linear equalities 
\[
l^{(1)}(\xi)=\cdots=l^{(n-k)}(\xi)=0
\]in $P\subset \frakt^*$. The relative interior $F^\circ$ of $F$ in $P$ is given by equalities and inequalities 
\[
l^{(1)}(\xi)= \cdots =l^{(n-k)}(\xi)=0, \ l^{(n-k+1)}(\xi)>0, \cdots , l^{(N)}(\xi)>0. 
\]
Let $\vphi:P^\circ\to \bbR$ be a symplectic potential, and 
consider the associated toric K\"ahler manifold $M_P$ and its moment map $\mu:M_P\to P$. 
Note that the Guillemin potential $\vphi_P$ can be extended as a continuous function on $P$, 
and its restriction to $F^\circ$ is a smooth function. 
In particular since $\vphi-\vphi_P$ is a smooth function on $P$, $\vphi$ can be extended as a continuous function on $P$ and  its restriction to $F^\circ$ is smooth. 
We put $M_P(F):=\mu^{-1}(F)$. 
It is known that : 
\begin{itemize}
\item $T_F^\perp$ acts trivially on $M_P(F)$, where $T_F^\perp$ is the subtorus generated by normal vectors $\nu^{(1)}, \cdots, \nu^{(n-k)}\in \frakt$ of $F$. Moreover the action of $T_F:=T/T_F^\perp$ on $M_P(F)$ is Hamiltonian. 
\item One can choose a moment map $\mu_F:M_P(F)\to {\rm Lie}(T_F)^*=\frakt_F^*$ so that the Delzant polytope $P_F:=\mu_F(M_P(F))$ is mapped to $F$ bijectively by the composition of the natural map $\pi_F^*:\frakt_F^*\to \frakt^*$ and an appropriate translation. 
\item The above bijection preserves the interior $P_F^{\circ}$ and relative interior $F^\circ$. 
\item $M_P(F)$ is a K\"ahler submanifold of $M_P$ of complex dimension $k$. This follows from the fact that $M_P(F)$ is a connected component of the fixed point set $(M_P)^{T_F^\perp}$ of the holomorphic $T_F^\perp$-action. 
\end{itemize}
From these facts and Remark~\ref{rem:toric to dually flat} the momentum map $\mu_F:M_P(F)\to P_F\subset \frakt_F^*$ on a toric K\"ahler manifold associates a dually flat space $(P_F^\circ, G_F={\rm Hess}(\vphi_F),\nabla^{\rm flat})$ for a symplectic potential $\vphi_F:P_F^{\circ}\to \bbR$. 



\begin{defn}\label{defn:boundary dually flat}
We call $(P_F^{\circ}, G_F, \nabla^{\rm flat})$ the {\it dually flat space of the boundary $F$.} 
The dually flat space of the boundary gives rise to the Bregman divergence 
\[
D_F(\cdot \| \cdot) : P_F^{\circ}\times P_F^{\circ}\to \bbR. 
\]
\end{defn}

\subsection{Continuity of the divergence on the boundary}
We discuss the continuity between $D(\cdot \| \cdot)$ and $D_F(\cdot \| \cdot)$. 
We rely on the continuity of the symplectic potential.
To simplify the explanation we extend normal vectors $\nu^{(1)}, \ldots, \nu^{(n-k)}$ of $F$ to a rational basis $\{e_1, \ldots, e_{k}, e_{k+1}=\nu^{(1)}, \ldots, e_n=\nu^{(n-k)}\}$ of $\frakt$ 
so that $F$ is represented by the equalities 
\[
\xi_{k+1}=\cdots=\xi_{n}=0
\] and $(\xi_1, \ldots, \xi_k)$ is a coordinate of $F$. 
Note that $\{e_{k+1}, \ldots, e_n\}$ is a basis of $\frakt_F^{\perp}$.  
The subspace $\frakt_F'$ generated by $\{e_1, \ldots, e_k\}$ is isomorphic to $\frakt_F={\rm Lie}(T/T_F^\perp)$ under the natural map $\pi_F$, and hence, the corresponding subtorus $T_F'$ has a structure of a covering space $\pi_F|_{T_F'} : T_F'\to T_F$. 
We identify $\frakt_F'$ and $\frakt_F$ via the differential of $\pi_F|_{T_F'}$, which makes no obstacle. 
In addition we may assume that $\pi_F^*(\mu_F(P_F))$ is equal to $F$ without any translation.  

Let $g^\vphi$ be a torus invariant Riemannian metric on the toric K\"ahler manifold $M_P$. 
Let $g^\vphi_F$ be the induced Riemannian metric on the K\"ahler submanifold $M_P(F)$. 
By Theorem~\ref{thm : toric Kahler metric} one has a symplectic potential $\vphi_F:P_F^\circ\times P_F^{\circ}\to\bbR$ of $g_F^\vphi$. 

\begin{lem}\label{lem:conti of potential}
$\vphi_F$ can be taken as $\vphi_F=\vphi\circ \pi_F^*|_{P_F^\circ}$, where $\vphi$ in the right hand side is considered as a smooth function on $F^\circ$. 
\end{lem}
\begin{proof}
We take a continuous section $s:P \to M_P$ so that the restriction $s|_{P^\circ}$ gives a Lagrangian section of $M_P^\circ$ and $(s|_{F^\circ})\circ \pi_F^*:P_F^\circ\to M_P(F)^{\circ}$ gives a Lagrangian section.  
Such a section can be taken by using a description of $M_P$ in symplectic cutting method. 
In this set-up it suffices to show that the Riemannian metric $g_F^\vphi$ on $M_P(F)^\circ=P_F^\circ\times T_F$ can be represented as  
\[
\begin{pmatrix}
G_F & 0 \\
0 & G_F^{-1}
\end{pmatrix} 
\]for $G_F={\rm Hess}(\vphi\circ \pi_F^*) : (P_F)^\circ\to {\rm sym}(\frakt_F^*)$. 

Recall that the $k\times k$ matrix $G_F$ is obtained as the inverse matrix of a positive definite matrix  $H_F$ whose $(i,j)$-th entry $(H_F)_{ij}$ is given by   
\[
(H_F)_{ij}=g^{\vphi}_F(\underline{e_i},\underline{e_{j}}) 
\]for $i,j=1,\ldots, k$, where $\underline{e_i}$ is the induced vector field of $e_i\in\frakt$ on $M_P(F)$. 
Note that the right hand side is a $T_F$-invariant function on $M_P(F)$, and hence, it defines a function on $M_F(P)^\circ/T_F=P_F^\circ$. 
Since $g_F^\vphi$ is the restriction of $g^\vphi$, we have the equalities 
\[
g^{\vphi}_F(\underline{e_i},\underline{e_{j}})=g^\vphi(\underline{e_i},\underline{e_{j}})=({\rm Hess}(\vphi))^{-1}_{ij}. 
\] It implies that 
\[
G_F=H_F^{-1}={\rm Hess}(\vphi)\circ\pi_F^*={\rm Hess}(\vphi\circ \pi_F^*) : P_F^\circ\to {\rm sym}(\frakt_F^*).
\]
\end{proof}

\begin{thm}\label{thm : conti divergence}
When points $\xi$ and  $\xi'$ in $P^\circ$ converge to $\eta$ and $\eta'$ in $F^\circ$,  
we have the equality 
\[
\lim_{\xi'\to \eta'}\lim_{\xi\to\eta}D(\xi \|\xi')=D_F((\pi_F^*)^{-1}(\eta)\|((\pi^*_F)^{-1})(\eta')). 
\]
\end{thm} 
\begin{proof}
By Proposition~\ref{prop :  Divergence of P} we have the formula 
\[
D(\xi \| \xi')=\frac{1}{2}\sum_{r=1}^{N}\left(l^{(r)}(\xi)\log\frac{l^{(r)}(\xi)}{l^{(r)}(\xi')}-(\xi-\xi')\cdot\nu^{(r)}\right)
 +(\xi'-\xi)\cdot (\grad(f)(\xi'))+f(\xi)-f(\xi'), 
 \]where $f:=\vphi-\vphi_P$. 
When $\xi$ tends to $\eta$ one has $l^{(r')}(\xi)\log (l^{(r')}(\xi)) \to 0 \ (r'=1,\ldots, n-k)$, and hence, we have the equality 
\[
\lim_{\xi\to \eta}D(\xi\|\xi')=\frac{1}{2}\sum_{r=n-k+1}^{N}\left(l^{(r)}(\eta)\log\frac{l^{(r)}(\eta)}{l^{(r)}(\xi')}-(\eta-\xi')\cdot\nu^{(r)}\right)
 +(\xi'-\eta)\cdot (\grad(f)(\xi'))+f(\eta)-f(\xi'). 
\]Since $l^{(r)}(\xi')>0$ and $l^{(r)}(\eta')>0$ for $r=n-k+1,\ldots, N$ one has  the equality 
\[
\lim_{\xi'\to\eta'}\lim_{\xi\to \eta}D(\xi\|\xi')=\frac{1}{2}\sum_{r=n-k+1}^{N}\left(l^{(r)}(\eta)\log\frac{l^{(r)}(\eta)}{l^{(r)}(\eta')}-(\eta-\eta')\cdot\nu^{(r)}\right)
 +(\eta'-\eta)\cdot (\grad f(\eta'))+f(\eta)-f(\eta'). 
\]By Lemma~\ref{lem:conti of potential} we have  the equality 
\[
f(\eta)=\vphi(\eta)-\vphi_P(\eta)=\vphi_F((\pi_F^*)^{-1}(\eta))-\vphi_{P_F}((\pi^*_F)^{-1}(\eta))
=:f_F((\pi_F^*)^{-1}(\eta)). 
\]By combining  the equality 
\[
(\eta'-\eta)\cdot ({\rm grad}(f)(\eta'))=((\pi_F^*)^{-1}(\eta')-(\pi_F^*)^{-1}(\eta))\cdot ({\rm grad}(f\circ\pi_F^*)((\pi_F^*)^{-1}(\eta')))
\]and Lemma~\ref{lem:conti of potential} for $P_F$ we have the desired equality 
\[
\lim_{\xi'\to \eta'}\lim_{\xi\to\eta}D(\xi \|\xi')=D_F((\pi_F^*)^{-1}(\eta)\|((\pi^*_F)^{-1})(\eta')). 
\]
\end{proof}

\begin{Rem}
The continuity of $D(\cdot \| \cdot)$ as two variable function cannot be expected. 
This is because of the discontinuity of $l^{(r)}(\xi)\log\frac{l^{(r)}(\xi)}{l^{(r)}(\xi')}$. 
\end{Rem}

\subsection{The generalized Pythagorean theorem}
The generalized Pythagorean theorem is a fundamental property of the divergence and important for applications to statistical inference. 
In this subsection we extend the generalized Pythagorean theorem to the boundary of Delzant polytopes. 

Hereafter we denote 
\[
D_F(\eta\|\eta'):=D_F((\pi_F^*)^{-1}(\eta)\|((\pi^*_F)^{-1})(\eta')) \quad (\eta, \eta'\in F^\circ)
\]for simplicity, and we put the  similar definition, $D_F'( \cdot \| \cdot):F^\circ\times P^\circ\to \bbR$, 
\[
D_F'(\eta \|\xi'):=\lim_{P^\circ\ni \xi\to \eta}D(\xi \| \xi') \quad (\eta\in F^\circ, \xi'\in P^\circ). 
\]
To state the Pythagorean theorem for $D_F$ and $D_F'$ we introduce notations for geodesics. 
For $\eta, \eta'\in F^\circ$ let $(\eta\eta')_{\nabla^{\rm flat}}$ be the geodesic from $\eta$ to $\eta'$ with respect to the trivial flat connection $\nabla^{\rm flat}$ on $\frakt^*$. 
For $\eta'\in F^\circ$ and $\xi\in P^\circ$ let $(\xi\eta')_{\nabla}$ be the geodesic $\gamma_\xi$ from $\xi$ with respect to the dual flat affine connection $\nabla=(\nabla^{\rm flat})^*$ on $P^\circ$ such that its limit goes to $\eta'\in F^\circ$, 
\[
\lim_{t\to \infty}\gamma_\xi(t)=\eta'.  
\]

\begin{thm}[The generalized Pythagorean theorem on the compactification]\label{thm:Pythagorean theorem}
For $\eta, \eta'\in F^\circ$ and $\xi''\in P^\circ$ if the tangential direction of $s((\xi''\eta')_{\nabla})$ has a limit which is perpendicular to 
$s((\eta\eta')_{\nabla^{\rm flat}})$ at $s(\eta')\in M_P$ with respect to the Riemannian metric $g^\vphi$ on $M_P$, then  we have the equality 
\[
D_F(\eta\| \eta')+D_F'(\eta'\|\xi'')=D_F'(\eta\|\xi''). 
\]
\end{thm}
\begin{proof}
We show 
\[
D_F(\eta\| \eta')+D_F'(\eta'\|\xi')-D_F'(\eta\|\xi')=
\lim_{P^\circ\ni\xi'\to \eta'}\lim_{P^\circ\ni\xi\to \eta}(D(\xi\| \xi')+D(\xi'\|\xi'')-D(\xi\|\xi''))=0. 
\]
By definition of the Bregman divergence and the direct computation we have the equalities 
\begin{eqnarray*}
D(\xi\| \xi')+D(\xi'\|\xi'')-D(\xi\|\xi'')&=&(x(\xi)-x(\xi'))\cdot(y(\xi')-y(\xi'')) \\ 
&=&(\xi-\xi')\cdot(y(\xi')-y(\xi'')) \\ 
&=&G(v(\xi'\xi)_{\nabla^{\rm flat}}, v(\xi''\xi')_{\nabla}), 
\end{eqnarray*}
where $v(\xi'\xi)_{\nabla^{\rm flat}}$ (resp. $v(\xi''\xi')_{\nabla})$ is the velocity vector of 
$(\xi'\xi)_{\nabla^{\rm flat}}$ (resp. $(\xi''\xi')_{\nabla})$) at $\xi'$ and
$G={\rm Hess}(\vphi)$ is considered as a Riemannian metric on $P^\circ$. 
Moreover since $s|_{P\circ}:P^\circ\to M_P^\circ$ is an isometric embedding one has the equalities  
\begin{eqnarray*}
D(\xi\| \xi')+D(\xi'\|\xi'')-D(\xi\|\xi'')&=&G(v(\xi'\xi)_{\nabla^{\rm flat}}, v(\xi''\xi')_{\nabla}) \\ 
&=&g^\vphi(ds(v(\xi'\xi)_{\nabla^{\rm flat}}), ds(v(\xi''\xi')_{\nabla})), 
\end{eqnarray*}
and as $\xi$ tends to $\eta$ in $P$ we have the limit 
\[
D(\xi\| \xi')+D(\xi'\|\xi'')-D(\xi\|\xi'')\to g^\vphi(ds(v(\xi'\eta)_{\nabla^{\rm flat}}), ds(v(\xi''\xi')_{\nabla})). 
\]
By the assumption we have the direction of the limit of $s(\gamma_{\xi''})$ which is perpendicular to 
$s((\eta\eta')_{\nabla})$ at $s(\eta')\in M_P$ with respect to $g^\vphi$, and hence,  we have the limit
\[
g^\vphi(ds(v(\xi'\eta)_{\nabla^{\rm flat}}), ds(v(\xi''\xi')_{\nabla}))\to0 \ {\rm as} \ \xi'\to \eta'. 
\]
Figure~\ref{fig1} is a schematic picture of the configuration. 
It implies 
\[
\lim_{P^\circ\ni\xi'\to \eta'}\lim_{P^\circ\ni\xi\to \eta}(D(\xi\| \xi')+D(\xi'\|\xi'')-D(\xi\|\xi''))=0. 
\]
\end{proof}

\begin{figure}[h]
\centering
\includegraphics[scale=0.37]{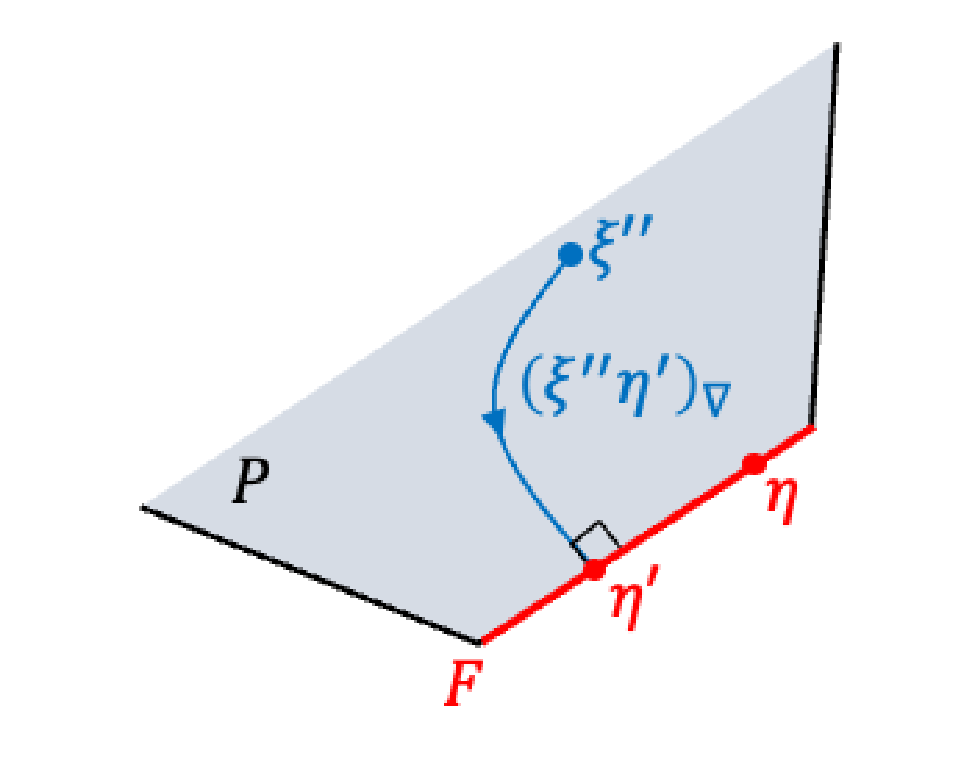}
\caption{Configuration of $\eta, \eta'$ and $\xi''$} \label{fig1}
\end{figure}

We also have an another version, which can be shown by the similar argument as for Theorem~\ref {thm:Pythagorean theorem}.  

\begin{thm}[The generalized Pythagorean theorem on the compactification~2]\label{thm:Pythagorean theorem2}
For $\eta\in F^\circ$ and $\xi, \xi'\in P^\circ$ if the $\nabla^{\rm flat}$-geodesic $(\eta\xi)_{\nabla^{\rm flat}}$ and the $\nabla$-geodesic $(\xi\xi')_{\nabla}$ are perpendicular at $\xi$, then we have the equality 
\[
D_F'(\eta \| \xi)+D(\xi \|\xi')=D'_F(\eta\|\xi'). 
\]
\end{thm}

\section{Examples}
In this section we demonstrate computations of several quantities in examples. 

\subsection{Isosceles triangle and $\bbC P^2$}
We first exhibit an example of our main theorem and verify our formula. 

\vspace{0.5cm}

\noindent
\underline{\bf Set-up. }

\medskip

Let $P$ be the triangle in $x_1x_2$-plane defined by inequalities 
\[
l^{(1)}(x_1, x_2)=x_1 \geq 0, \ l^{(2)}(x_1, x_2)=x_2 \geq 0, \ l^{(3)}(x_1, x_2)=1-x_1-x_2 \geq 0. 
\]
This is the image of the map $\mu : M_P=\bbC P^2\to \bbR^2$ defined by 
\[
\mu([z_1:z_2:z_3]):=\frac{1}{|z_1|^2+|z_2|^2+|z_3|^2}(|z_1|^2, |z_2|^2), 
\]which is the moment map for the Hamiltonian $T=(S^1)^2$-action 
\[
(t_1, t_2)\cdot[z_1:z_2:z_3]=[t_1z_1:t_2z_2:z_3]
\]on $\bbC P^2$. 
This $P$ can be also seen as a family of probability density functions $\Theta_P$ on the finite set $[3]=\{1,2,3\}$, 
\[
p(\xi|k)= 
\begin{cases}
\hspace{0.8cm} \xi_1 \qquad (k=1) \\ 
\hspace{0.8cm} \xi_2 \qquad (k=2) \\
1-\xi_1-\xi_2 \hspace{0.2cm} (k=3)
\end{cases}
\]for $\xi=(\xi_1, \xi_2)\in P$. 
$\Theta_P$ is often called the {\it categorical distribution} of three points. 

\vspace{0.5cm}

\noindent
\underline{\bf Potential, dual coordinate and divergence. }

\medskip

The Guillemin potential $\vphi_P:P\to \bbR$ is given by\[
\vphi_P(x_1, x_2)=\frac{1}{2}\left(x_1\log x_1+x_2\log x_2+(1-x_1-x_2)\log(1-x_1-x_2)\right)
\]whose Hessian is 
\[
G_P(x_1, x_2)=\frac{1}{2}
\begin{pmatrix}
\frac{1}{x_1}+\frac{1}{1-x_1-x_2} & \frac{1}{1-x_1-x_2} \\
\frac{1}{1-x_1-x_2} & \frac{1}{x_2}+\frac{1}{1-x_1-x_2}
\end{pmatrix}. 
\]
This is the Fisher metric of $\Theta_P$ and 
the corresponding toric K\"ahler metric is the Fubini-Study metric on $M_P=\bbC P^2$. 

The dual coordinate of $x=(x_1, x_2)$ and its inverse map are given by 
\[
y=(y_1, y_2)={\rm grad}(\vphi_P)=\left(\frac{1}{2}\log\frac{x_1}{1-x_1-x_2}, \ \frac{1}{2}\log\frac{x_2}{1-x_1-x_2}\right)
\]and 
\[
x=(x_1, x_2)=\left(\frac{e^{2y_1}}{1+e^{2y_1}+e^{2y_2}}, \ \frac{e^{2y_2}}{1+e^{2y_1}+e^{2y_2}}\right). 
\]
By Proposition~\ref{prop : Divergence of P}  and comments in Subsection~\ref{subsec : prob densi func} the divergence is given by 
\[
D_P(x \| x')=\frac{1}{2}\left(x_1\log\frac{x_1}{x_1'}+x_2\log\frac{x_2}{x_2'} +(1-x_1-x_2)\log\frac{1-x_1-x_2}{1-x_1'-x_2'}\right). 
\]

\vspace{0.5cm}

\noindent
\underline{\bf Behavior on the boundary. }

\medskip

Let $F$ be a boundary of $P$ defined by 
\[
l^{(3)}(x_1, x_2)=1-x_1-x_2=0.
\]
The corresponding submanifold $\mu^{-1}(F)$ is 
\[
M_P(F)=\{[z_1:z_2:z_3]\in \bbC P^2 \ | \ z_3=0 \}
\]whose stabilizer subgroup $T_F^\perp$ is the diagonal subgroup in $T$.  
An isomorphism $T_F=T/T_F^\perp\to S^1$ is given by $[t_1, t_2]\mapsto t_1^{-1}t_2$. 
$M_P(F)$ is diffeomorphic to $\bbC P^1$, and the  induced K\"ahler metric on $M_P(F)$ is the Fubini-Study metric on $\bbC P^1$ with the Guillemin potential $\vphi_{P_F}: P_F\to \bbR$, 
\[
\vphi_{P_F}(\eta)=\frac{1}{2}\left(\eta\log \eta + (1-\eta)\log (1-\eta)\right)
\] for $P_F=[0,1]$. 
 The affine map $\bbR\to \bbR^2$, $\eta\mapsto (\eta,1-\eta)$ gives an identification between $P_F=[0,1]$ and $F$. 
In this case Theorem~\ref{thm : conti divergence} can be seen as 
\begin{eqnarray*}
\lim_{x_1'+x_2'\to 1}\lim_{x_1+x_2\to 1}D_P(x \| x')&=&\frac{1}{2}\left(x_1\log\frac{x_1}{x_1'}+x_2\log\frac{x_2}{x_2'}\right) \\ 
&=& \frac{1}{2}\left(x_1\log\frac{x_1}{x_1'}+(1-x_1)\log\frac{1-x_1}{1-x_1'}\right) \\ 
&=& D_F(x_1 \| x_1')
\end{eqnarray*} for $(x_1, x_2), (x_1', x_2')\in F^\circ$. 

\vspace{1cm}
\noindent 
\underline{\bf Generalized Pythagorean theorem. }

\medskip

We first investigate geodesics on the dually flat space $(P^\circ, G_P, \nabla^{\rm flat})$. 
$\nabla^{\rm flat}$ (resp. $\nabla$)-geodesics are nothing other than the affine functions with respect to $x=(x_1, x_2)$ (resp. $y=(y_1, y_2)$). 
Namely for $\xi:=(a,b)\in P^\circ$ and $(v_1, v_2)\in \bbR^2$ an affine function 
\[
y(t)=(y_1(t), y_2(t))=\left(\frac{1}{2}\log \frac{a}{1-a-b}+tv_1, \frac{1}{2}\log \frac{b}{1-a-b}+tv_2\right) \quad (t\in\bbR)
\]gives a $\nabla$-geodesic $\gamma:\R\to P^\circ$, 
\[
\gamma(t)=x(y(t))=\left(\frac{ae^{2tv_1}}{1-a-b+ae^{2tv_1}+be^{2tv_2}}, \frac{be^{2tv_2}}{1-a-b+ae^{2tv_1}+be^{2tv_2}}\right)
\]starting from $\xi$. 
We assume $v_1, v_2>0$ for simplicity. 
We have the limit of $\gamma(t)$ as follows. 
\begin{enumerate}
\item[Case~1.] $v_1>v_2$ : $\displaystyle\lim_{t\to\infty}\gamma(t)=(1,0)$
\item[Case~2.] $v_1<v_2$ : $\displaystyle\lim_{t\to\infty}\gamma(t)=(0,1)$ 
\item[Case~3.] $v_1=v_2$ : $\displaystyle\lim_{t\to\infty}\gamma(t)=\left(\frac{a}{a+b}, \frac{b}{a+b}\right)$
\end{enumerate}

\begin{figure}[h]
\centering
\includegraphics[scale=0.4]{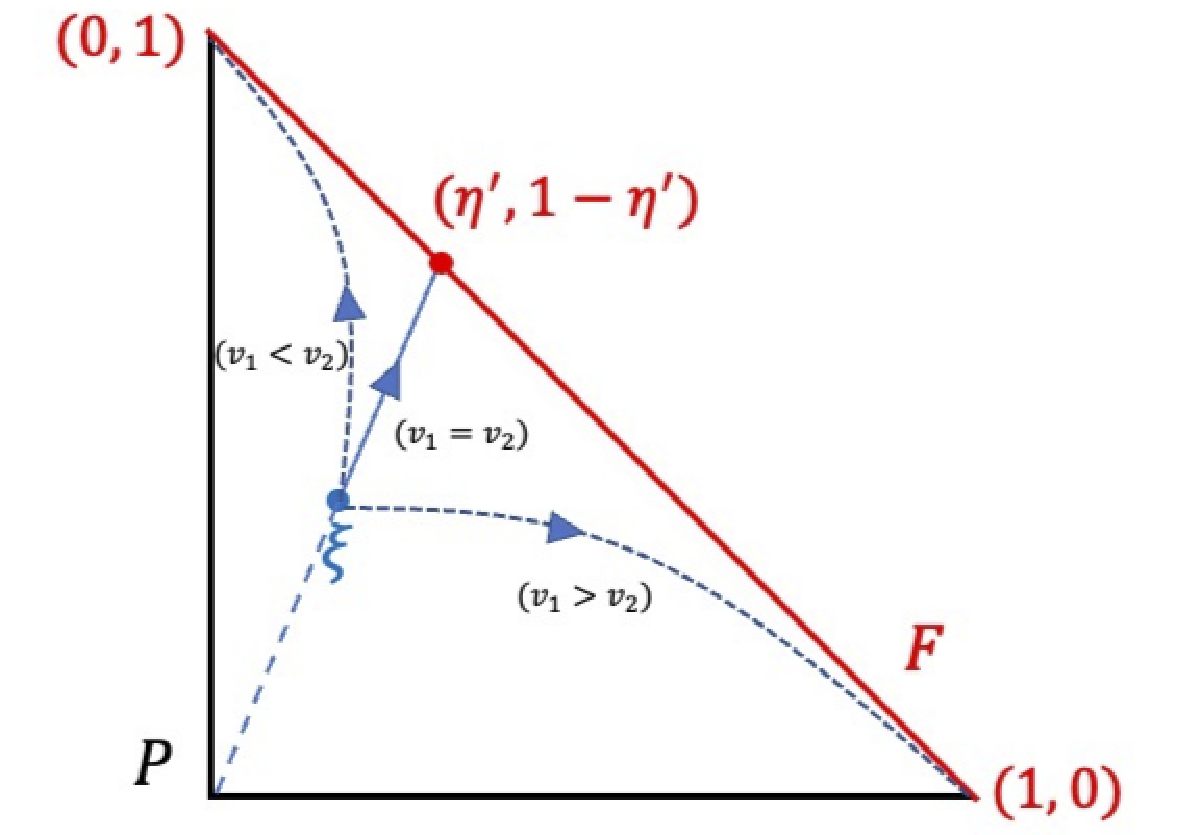}
\caption{Geodesics from $\xi$} \label{fig3}
\end{figure}

All these three cases give points in $F$, however, only the Case~3 gives a point in $F^\circ$ and we study this case in relation to Theorem~\ref{thm:Pythagorean theorem}. Put $v=v_1=v_2$. 
In fact if $v_1=v_2$, then $\gamma(t)$ lies on the line $bx_1=ax_2$. 
We take a section $s:P\to \bbC P^2$, 
\[
s(x_1, x_2):=[\sqrt{x_1} : \sqrt{x_2} : \sqrt{1-x_1-x_2} ]. 
\]This gives a smooth curve
\begin{eqnarray*}
&s(\gamma(t))& \\ 
&=& \left[\sqrt{\frac{ae^{2tv}}{1-a-b+ae^{2tv}+be^{2tv}}} 
: \sqrt{\frac{be^{2tv}}{1-a-b+ae^{2tv}+be^{2tv}}} : 
\sqrt{\frac{1-a-b}{1-a-b+ae^{2tv}+be^{2tv}}} \right] 
\\ 
&=&[\sqrt{a} : \sqrt{b} : e^{-tv}\sqrt{1-a-b}]. 
\end{eqnarray*}
Take $\eta\in P_F$ and put $\eta':=\frac{a}{a+b}$,  which correspond to $(\eta, 1-\eta)$ and $(\eta', 1-\eta')=\left(\frac{a}{a+b}, \frac{b}{a+b}\right)\in F^\circ$. 
We also have $s(\eta,1-\eta)=[\sqrt{\eta} :\sqrt{1-\eta}:0]$ and $s(\eta', 1-\eta')=[\sqrt{a}:\sqrt{b}:0] \in M_P(F)=\bbC P^1\subset M_P=\bbC P^2$. 
The limit of the tangential direction of $s(\gamma(t))$ is perpendicular to $M_P(F)$ with respect to the Fubini-Study metric. 
This observation implies that three points $(a,b), (\eta, 1-\eta)$ and  $(\eta', 1-\eta')$ satisfy the assumption in Theorem~\ref{thm:Pythagorean theorem}. 
For these points we have the equalities 
\begin{eqnarray*}
D_F(\eta\|\eta')&=&\frac{1}{2}\left( \eta\log\frac{\eta}{\eta'}+(1-\eta)\log\frac{1-\eta}{1-\eta'}\right) \\ 
&=&\frac{1}{2}\left(\eta\log\frac{\eta(a+b)}{a}+(1-\eta)\log\frac{(1-\eta)(a+b)}{b}\right) \\ 
&=& \frac{1}{2}\left(\eta\log\frac{\eta}{a} +(1-\eta)\log\frac{(1-\eta)}{b}+\log(a+b)\right), 
\end{eqnarray*}
\begin{eqnarray*}
D_P(\gamma(t) \| \xi) &=& \frac{1}{2}\left( \frac{ae^{2tv}}{1-a-b+ae^{2tv}+be^{2tv}} \log \frac{e^{2tv}}{1-a-b+ae^{2tv}+be^{2tv}} \right. \\ 
&& \hspace{2.5cm} \left. +\frac{be^{2tv}}{1-a-b+ae^{2tv}+be^{2tv}}\log\frac{e^{2tv}}{1-a-b+ae^{2tv}+be^{2tv}} \right. \\ 
&& \hspace{2.5cm} \left. +\frac{1-a-b}{1-a-b+ae^{2tv}+be^{2tv}}\log \frac{1}{1-a-b+ae^{2tv}+be^{2tv}}\right)\\ 
&\rightarrow& \frac{1}{2}\left(\frac{a}{a+b}\log\frac{1}{a+b}+\frac{b}{a+b}\log\frac{1}{a+b} \right) \quad  (t\to \infty) \\ 
&=&-\frac{1}{2}\log(a+b) \\ 
&=&D_F'(\eta'\|\xi)
\end{eqnarray*}
and 
\[
D_F'(\eta\|\xi)=\frac{1}{2}\left(\eta\log \frac{\eta}{a}+(1-\eta)\log\frac{1-\eta}{b}\right). 
\]Summarizing we have the equality
\[
D_F'(\eta\|\xi)=D_F(\eta\|\eta')+D_F'(\eta'\|\xi)
\]as in Theorem~\ref{thm:Pythagorean theorem}.

\subsection{Non-compact example}

The notion of torification does not need the compactness of the toric manifold. 
In fact as it is studied in \cite{Molitor} several important examples from information geometry have non-compact regular torifications. 
Here we show that the compactness is not also essential for our main theorems by giving a (relatively trivial) non-compact example. 

Consider the standard $S^1$-action on the complex plane $\bbC$. 
This can be seen as a non-compact toric manifold with the moment map 
\[
\mu:\bbC\to\bbR, \quad z\mapsto \frac{|z|^2}{2}
\]whose moment map image is $\bbR_{\geq 0}$. 
Consider the potential function $\vphi : \bbR_{>0}\to \bbR$, \ $\xi\mapsto \xi\log \xi$ with its Hessian 
$G={\rm Hess}(\vphi)=\frac{1}{\xi}$. 
This gives the Bregman divergence 
\[
D(\xi\|\xi')=\xi\log\frac{\xi}{\xi'}+\xi'-\xi. 
\]
The potential $\vphi$ gives a K\"ahler structure and a Riemannian metric 
\[
\begin{pmatrix}
\frac{1}{\xi} & 0 \\
0 & \xi
\end{pmatrix}
=\frac{d\xi^2}{\xi}+\xi dt^2
\]on $\bbC^\circ=\bbC\setminus\{0\}\cong \bbR_{+}\times S^1$.
This is the Euclidean metric on $\bbC=\bbR^2$ with respect to the polar coordinate 
\[
z=\sqrt{\xi}e^{it}. 
\] 
As it is shown in \cite{Molitor}, this toric manifold is the regular torification of the dually flat space corresponding to the {\it Poisson distribution}. 

Let $P$ be an $n$-dimensional Delzant polytope with corresponding symplectic toric manifold $M_{P}$ with the moment map $\mu_{P}:M_{P}\to P$. 
Consider the Guillemin potential $\vphi_{P}:P\to \bbR$ and the associated Bregman divergence $D_P$.  
Let $\widetilde M$ (resp. $\widetilde P$) be the product of $M_{P}$ and $\bbC$ (resp. $P$ and $\bbR_{\geq 0}$), 
$\widetilde M= M_{P}\times \bbC$ (resp. $\widetilde P=P\times \bbR_{\geq 0}$). 
Note that the boundary is given by 
\[
\partial \widetilde P=(\partial P\times \bbR_{> 0})\cup (P^\circ\times \{0\}).
\] 
There are the associated moment map $\widetilde \mu=\mu_P\times\mu : \widetilde M\to \widetilde P$ and potential $\tilde \vphi=\vphi_P+\vphi : \widetilde P\to\bbR$.
The associated divergence $\widetilde D:\widetilde P\times \widetilde P\to\bbR$ is given by  
\[
\widetilde D((\xi_1, \xi_2)\|(\xi_1', \xi_2'))=D_P(\xi_1\|\xi_1')+D(\xi_2\|\xi_2'), 
\]which has natural extension to the boundary $\partial \widetilde P$. 
For this divergence one can check the extended Pythagorean theorem 
\[
\widetilde D((\eta,\xi_2)\|(\xi_1, \xi_2'))=\widetilde D((\eta, \xi_2)\|(\xi_1,\xi_2))+\widetilde D((\xi_1, \xi_2)\|(\xi_1, \xi_2'))
\]for $(\eta,\xi_2)\in \partial P\times \bbR_{>0}$ and $(\xi_1, \xi_2), (\xi_1,\xi_2')\in  P^\circ\times\bbR_{>0}$. 
We also have  the equality  
\[
\widetilde D((\xi_1, 0)\|(\xi_1',\xi_2))=\widetilde D((\xi_1, 0)\|(\xi_1,\xi_2))+\widetilde D((\xi_1, \xi_2)\|(\xi_1', \xi_2))
\]for $(\xi_1, 0)\in P^\circ\times\{0\}$ and $(\xi_1',\xi_2), (\xi_1,\xi_2)\in P^\circ\times \bbR_{>0}$. 
See Figure~\ref{fig4} for the configurations of points in the above equalities. 
\begin{figure}[h]
\centering
\includegraphics[scale=0.3]{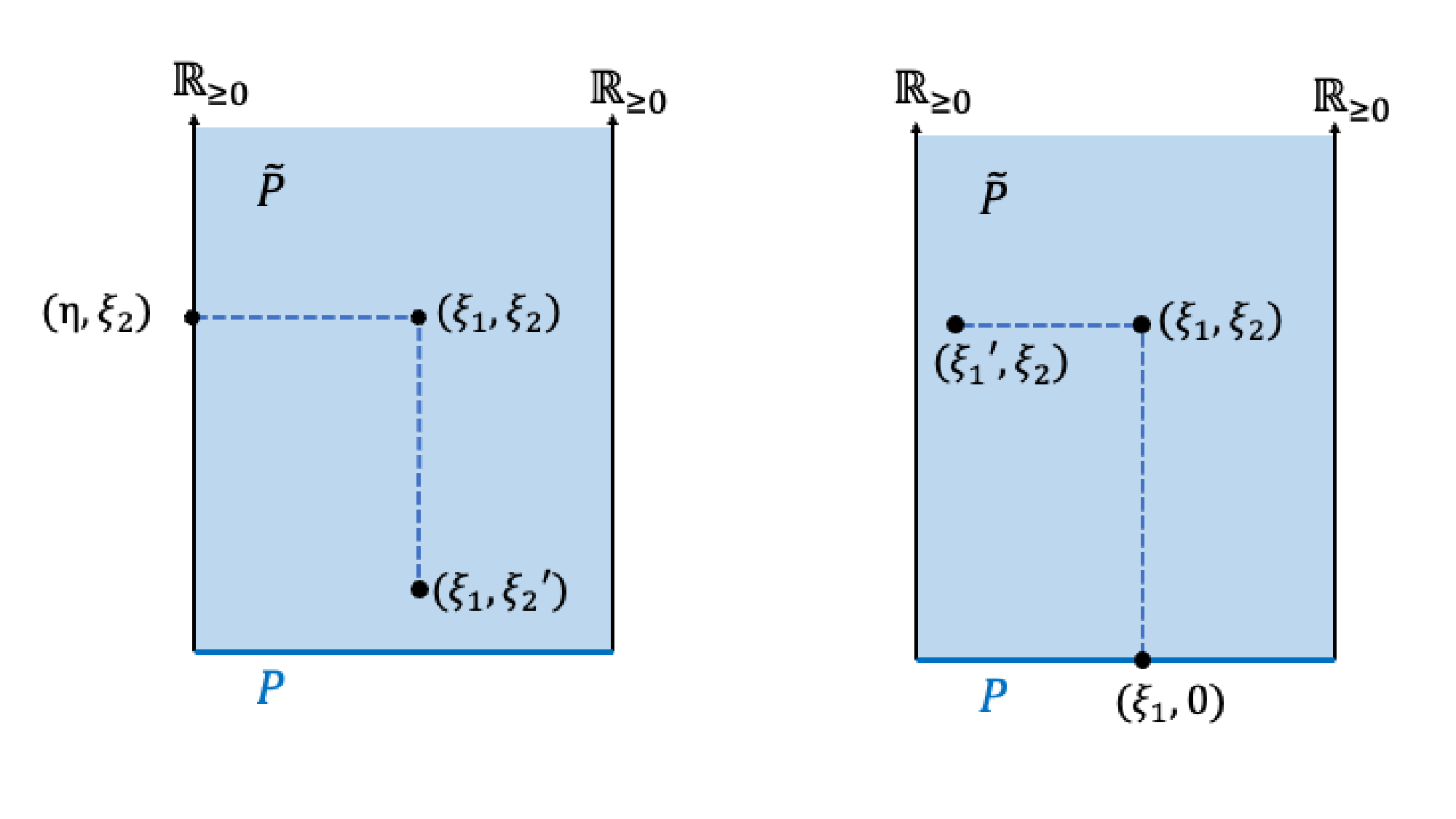}
\caption{Configurations in $\widetilde P$} \label{fig4}
\end{figure}

\bibliography{refdelbeppufujimitsu}

\end{document}